\documentclass[11pt]{amsproc}

\usepackage{amsfonts,amsmath,amsthm}
\usepackage{amssymb}
\usepackage{tikz}
\usepackage{url}
\usepackage{hyperref}
\usepackage[utf8]{inputenc}
\usepackage[english]{babel}
\usepackage{graphicx}

\newtheorem{theorem}{Theorem}[section]
\newtheorem{lemma}[theorem]{Lemma}

\newtheorem{proposition}[theorem]{Proposition}

\newtheorem{definition}[theorem]{Definition}

\theoremstyle{remark}
\newtheorem{nota}[theorem]{}

\newtheorem{remark}[theorem]{Remark}

\newtheorem*{claim*}{Claim}

\def\R{\mathbb R}
\def\C{\mathbb C}

\def\N{\mathbb N}
\def\Q{\mathcal Q}
\def\B{\mathcal B}
\def\NN{\mathcal N}
\def\P{\mathbb P}
\def\Id{\operatorname{Id}}
\def\GL{\operatorname{GL}}
\def\PGL{\operatorname{PGL}}

\def\Iso{\operatorname{Iso}}

\def\ch{\operatorname{ch}}
\def\norma{\|\cdot\|}
\def\nX{\|\cdot\|_X}
\def\RnX{(\R^n,\|\cdot\|_X)}
\def\XnX{(X,\|\cdot\|_X)}

\def\RnY{(\R^n,\|\cdot\|_Y)}

\def\nY{\|\cdot\|_Y}
\def\nZ{\|\cdot\|_Z}

\def\nnn{\big|\!\big|\!\big|}

\def\nnnn{\big|\!\big|\!\big|\cdot\big|\!\big|\!\big|}
\def\coc{\negthinspace\sim}
\def\opX{\tilde{X}}
\def\wopX{\widetilde{X}}
\def\e{\varepsilon}

\def\bij{\longleftrightarrow}

\title{The Banach space of quasinorms on a finite-dimensional space}

\begin{document}

\author
{J.~Cabello~S\'anchez, D.~Morales~Gonz\'alez}

\thanks
{Departamento de Matem\'{a}ticas and Instituto de Matem\'aticas. 
Universidad de Extremadura,
Avda. de Elvas s/n, 06006, Badajoz, Spain \\
email:\ coco@unex.es (corresponding author), danmorg@unex.es}  



\begin{abstract}
We show that the set of continuous quasinorms on a finite-dimensional linear space, after
quotienting by the dilations, has a natural structure of Banach space. 
Our main result states that, given a finite-dimensional vector 
space $E$, the pseudometric defined in the set of continuous quasinorms 
$\mathcal{Q}_0=\{\|\cdot\|:E\to\mathbb{R}\}$ as 
$$d(\|\cdot\|_X,\|\cdot\|_Y)=\min\{\mu:\|\cdot\|_X
\leq\lambda\|\cdot\|_Y\leq\mu\|\cdot\|_X\text{ for some }\lambda \}$$
induces, in fact, a complete norm when we take the obvious quotient 
$\mathcal{Q}=\mathcal{Q}_0/\!\sim$ and define the appropriate operations on $\mathcal{Q}$. 

We finish the paper with a little explanation of how this space and the Banach-Mazur 
compactum are related. 
\end{abstract}

\maketitle

\section{Introduction}

Our main goal in this short paper is to show that the set of continuous quasinorms defined 
on $\R^n$ for some $n\geq 2$ has a, somehow, canonical structure of Banach space 
after quotienting by the proportional quasinorms.

For this to make sense, we first need to endow this set with a vector space structure --this 
will be done by means of something that everyone can expect to represent the 
mean of two quasinorms: $\sqrt{\nX\nY}$ for each pair of quasinorms $\nX, \nY$. 
Once the mean is given, we just need to choose the element of the space 
which will play the rôle of the origin in order to determine a vector space structure, 
in the present paper we have chosen $(\R^n,\norma_2)$. 
Of course, this may seem anything but canonical. On the bright side, the choice 
of an origin will not affect any property of the newly defined vector (or Banach) space. 
For example, we may consider $C[0,1]$ endowed with the scalar multiplication 
$(\lambda\star f)(x)=\lambda(f(x)-1)$ and the addition $(f\oplus g)(x)=f(x)+g(x)-1$ 
--of course, the same can be done with any other function in $C[0,1]$ instead 
of 1. Now, if we define a norm in $(C[0,1],\oplus,\star)$ as $\|f\|=\max\{|f(x)-1|\}$ 
then we have a Banach space structure $(C[0,1],\oplus,\star,\norma)$ that is 
indistinguishable from the usual $(C[0,1],+,\cdot,\norma_{\infty})$, in the sense 
that the map 
$$(C[0,1],+,\cdot,\norma_{\infty})\to (C[0,1],\oplus,\star,\norma),\quad f\mapsto f+1$$ 
is a linear isometry. What we have done is equivalent to considering the affine 
structure of $C[0,1]$ and taking two different choices for the origin. 
This is doable because every norm gives a translation invariant metric. 

Once the operations are given, we have to define the norm. This idea is not 
ours, but taken from A.~Khare's preprint \cite{Apoorva}. Given two continuous 
quasinorms $\nX, \nY$, the distance between them is defined as 
$$d(\nX,\nY)=\min\{\mu:\nX\leq\lambda\nY\leq\mu\nX\text{\ for\ some\ }\lambda>0\},$$
where the order relation is the pointwise order: 
$\nX\leq\lambda\nY$ means $\|x\|_X\leq \lambda\|x\|_Y$ for every $x\in\R^n$. 
Of course, two quasinorms are proportional if and only if the distance between them 
is 1, so we must take the reasonable quotient 

\vspace{3 mm}

$$\spadesuit\quad\nX\sim\nY \text{ if and only if }
\nX=\lambda\nY \text{ for some }\lambda\in(0,\infty)$$  

\vspace{3 mm}
\noindent to make $d$ an actual (multiplicative) metric. So, defining 
\begin{equation}\label{eqnlm}
d([\nX],[\nY])=\min\{\mu:\nX\leq\lambda\nY\leq\mu\nX\text{\ for\ some\ }\lambda>0\} 
\end{equation}
we have a distance between the equivalence classes of quasinorms that turns 
out to induce a norm when we endow $\{\text{Continuous quasinorms on }\R^n\}/\coc$ with 
the above explained operations. 

This paper is far from being the first one in which the sets of (quasi) norms 
are endowed with some structure. The best known structure given to the set of 
norms on a finite-dimensional space is the Banach-Mazur pseudometric defined as 
\begin{equation}\label{eqBM}
d(\nX,\nY)=\inf\{\|T\|\|T^{-1}\|\},
\end{equation}
where $T$ runs over the set of linear isomorphisms $T:\RnX\to\RnY.$
It is well known that, after taking the appropriate quotient, this pseudometric 
turns out to be a metric that makes the space to be compact --there is still 
significant interest on this topic, see, e.g., \cite{Antonyan, Szarek, Youssef}. 

The present paper is neither the first one about, say, mixing pairs of 
norms to obtain something new. In this setting, interpolation of (quasi) normed 
spaces --or even more general spaces-- has been the main topic for at least half 
a century, see \cite{kalton1992factorization, Peetre, Slodkowski}. For the reader 
interested in interpolation, we suggest~\cite{BL} and the very interesting 
\cite{semmes}. A nice paper on interpolation in quasinormed spaces is~\cite{Silva}.
To the best of our knowledge, this paper is the first where someone considers the 
kind of interpolation that we have in Definition~\ref{definterp}, that is 
$$\norma_{(X,Y)_\theta}=\nX^\theta\nY^{1-\theta}.$$
There is a very good reason to avoid this kind of interpolation in the \textit{normed} 
space literature. Namely, in Remark~\ref{Nocvx} we provide an example to show that 
the mean of a pair of norms on $\R^2$ does not need to be a norm but a quasinorm. 

\section{Notations and preliminary results}

We will consider some positive integer $n$ fixed throughout the paper. 
Every vector space will be over $\R$; observe that any $\C^n$ can be seen 
as $\R^{2n}$. Moreover, we will consider from now on the vector space 
$\R^n$ endowed with its only topological vector space structure, i.e., 
the one given by $\norma_2$. 

\begin{definition}
A map $\norma:\R^n\to[0,\infty)$ is a quasinorm if the following conditions hold: 
\begin{enumerate}
\item $\|x\|=0$ if and only if $x=0.$
\item $\|\lambda x\|=|\lambda|\|x\|$ for every $\lambda\in\R, x\in\R^n.$
\item There exists $k>0$ such that $\|x+y\|\leq k(\|x\|+\|y\|)$ for every $x,y\in\R^n$.
\end{enumerate}
If the map $\norma$ is continuous then we say that it is a continuous quasinorm. 
If $k$ can be chosen to be 1, then $\norma$ is a norm and it is continuous. 
\end{definition} 

\begin{nota}\label{BXSX}
As is customary, given a quasinormed space $(\R^n,\nX)$, we will denote its unit 
(closed) ball as $B_X$, its unit sphere as $S_X$.  
\end{nota}

\begin{definition}\label{balanced}
Some subset $B\subset \R^n$ is bounded if, for every 
neighbourhood $U$ of 0 there is $M\in(0,\infty)$ such that $B\subset MU$. 
$B\subset \R^n$ is balanced when $\lambda B\subset B$ for every $\lambda\in[-1,1]$. 
\end{definition}

\begin{definition}\label{minkowski}
For $B\subset \R^n$, the Minkowski functional of $B$ 
is $\rho_B(x)=\inf\{\lambda\in(0,\infty):x\in\lambda B\}$. 
\end{definition}

It is quite well known that the quasinorms on a topological vector space are in 
correspondence with the bounded, balanced neighbourhoods of the origin, see the 
beginning of Section~2 in~\cite{KaltonThree}, and, for a proof of such a key 
result the reader may check~\cite[Theorem~4]{Hyers}. The version that we will 
use is the following, where we use that $\norma_2$ gives the only topological 
vector space structure to $\R^n$ and $B_2$ denotes the Euclidean unit ball of $\R^n$: 

\begin{theorem}\label{rho}
The Minkowski functional $\rho_B$ of a given subset $B\subset \R^n$ is a 
quasinorm if and only if $B$ fulfils the following: 
\begin{itemize}
\item $B$ contains $\e B_2$ for some $\e>0$. 
\item For every $\lambda\in[-1,1]$ one has $\lambda B\subset B$ -- i.e., $B$ is balanced. 
\item $B$ is contained in $MB_2$ for some $M>0$. 
\end{itemize}
In this case, $\rho_B$ is a continuous quasinorm if and only if 
$B$ is closed. 
Moreover, $\rho_B$ is a norm if and only if the above hold and 
$\frac {x+y}2\in B$ for any pair $x,y\in B.$
\end{theorem}

We could even replace the first and third items in~\ref{rho} by \\ 
``If $B'$ is the unit ball of some quasinorm on $\R^n$ 
then there are $\e, M>0$ such that $\e B\subset B' \subset MB$."
Observe that this implies that the constants $\lambda, \mu$ in 
(\ref{eqnlm}) actually exist. 

We will deal in this note with $\Q_0=\{\text{Continuous quasinorms\ 
defined\ on\ }\R^n\}$ and $\Q=\Q_0/\coc$, where 
two quasinorms are equivalent if and only if they are proportional, endowed 
with the multiplicative distance on $\Q$ defined in~\cite{Apoorva} by 
A.~Khare and given by 
\begin{equation}\label{defdis}
d([\nX],[\nY])=\min\{\mu:\nX\leq\lambda\nY\leq\mu\nX\text{\ for\ some\ }\lambda>0\}. 
\end{equation}

In the same paper, it is shown that $d$ endows  
$\NN=\{\text{Norms\ defined\ on\ }\R^n\}/\coc$ 
with a complete metric space structure. To keep the 
notations consistent, we will write $\NN_0$ for $\{\text{Norms\ defined\ on\ }\R^n\}$. 
The infimum in (\ref{defdis}) exists because in $\R^n$ every pair of 
continuous quasinorms are equivalent and, moreover, by the continuity 
of the quasinorms, it is pretty clear that the minimum is attained. A 
nice feature of Khare's distance is that, in $\R^2$, it 
distinguishes the $\max$-norm from $\norma_1$. In some sense, these norms 
are as different as two norms can be, but the usual distances between norms, 
such as the Banach-Mazur or the Gromov-Hausdorff, make them indistinguishable. 

\section{The main result}

Throughout this section, we will only consider continuous quasinorms.

Our first goal is to show that $d$ is actually a multiplicative distance on $\Q$. 
For this, the following lemma will be useful. 

\begin{lemma}\label{soloy} 
Take any pair of quasinorms $\nX,\nY$, $\lambda>0$ and $\mu\geq 1$ such that 
$\nX\leq\lambda\nY\leq\mu\nX$. Then, $\mu$ is minimal if and only if 
both $S_X\cap\lambda^{-1} S_Y$ and $\lambda^{-1} S_Y\cap\mu^{-1} S_X$ are non-empty. 

Moreover, the distance between $[\nX]$ and $[\nY]$ is $\mu$ if and only 
if there are representatives $\nX$ and $\nY$ such that:
\begin{enumerate}
\item One has $\nX\leq\nY\leq\mu\nX$.
\item There are $x\in S_X\cap S_Y$ and $y\in S_Y\cap\mu^{-1}S_X$.
\end{enumerate}
In particular, the distance $\mu$ is always attained. 
\end{lemma}

\begin{proof}
The chain of inequalities in the statement is equivalent to the chain of inclusions 
$\mu^{-1}B_X\subset\lambda^{-1}B_Y\subset B_X$, so suppose that 
$S_X\cap\lambda^{-1} S_Y=\emptyset$. The distance between the compact sets 
$\lambda^{-1}B_Y$ and $S_X$ is attained, so if they do not meet, then the distance 
between them is strictly positive and we can multiply the sets $\lambda^{-1}B_Y$ 
and $\mu^{-1}B_X$ by $1+\varepsilon$ for some $\varepsilon>0$ and the contentions 
are still fulfilled. So, if we define $\mu'=\dfrac{\mu}{1+\varepsilon}$ we obtain 
$\mu'^{-1}B_X\subset(1+\varepsilon)\lambda^{-1}B_Y\subset B_X$. So, 
$\mu$ would not be minimal because $\mu'<\mu$. The case 
$\lambda^{-1} S_Y\cap\mu^{-1} S_X=\emptyset$ is analogous. 

The other implication is clear. 
\end{proof}

\begin{proposition}
The function $d$ defined in~{\em(\ref{defdis})} is a multiplicative distance. 
\end{proposition}

\begin{proof}
We need to show that $d$ fulfils the following: 
\begin{enumerate}
\item $d([\nX],[\nY])=1$ if and only if $\nX$ and $\nY$ are proportional. 
\item $d([\nX],[\nY])=d([\nY],[\nX])$. 
\item $d([\nX],[\nY])\leq d([\nX],[\nZ])d([\nZ],[\nY])$. 
\end{enumerate}

The first item is obvious since we have taken the quotient exactly for this.  \\
It is clear that 
\begin{equation}
\begin{split}
d([\nX],[\nY])=&\min\{\mu:\nX\leq\nY\leq\mu\nX\} \\
=&\min\{\mu:\nX\leq\nY\leq\mu\nX\leq\mu\nY\} \\
=&\min\{\mu:\nY\leq\mu\nX\leq\mu\nY\} \\
=&d([\nY],[\nX]),
\end{split}
\end{equation}
so the second item also holds. 

For the third item, let $\mu=d([\nX],[\nZ]), \mu'=d([\nZ],[\nY])$. 
There exist $\lambda, \lambda'$ such that 
$$\nX\leq\lambda\norma_Z\leq\mu\nX\text{\quad and\quad }\norma_Z\leq\lambda'\nY\leq\mu'\norma_Z.$$ 
Joining these inequalities, we obtain 
$\nX\leq\lambda\nZ\leq\lambda\lambda'\nY\leq \lambda\mu'\nZ\leq\mu\mu'\nX.$ 
This readily implies that $d([\nX],[\nY])\leq\mu\mu'=d([\nX],[\nZ])d([\nZ],[\nY])$. 
\end{proof}

In order to define the operations in $\Q$, we need the following: 

\begin{definition}\label{definterp}
Let us denote $X=(\R^n,\nX)$ and $Y=(\R^n,\nY)$ and let $\theta\in[0,1]$. 
We will call the space $\R^n$ endowed with the quasinorm 
$$\norma_{(X,Y)_\theta}=\nX^\theta\nY^{1-\theta}$$ 
the interpolated space between $X$ and $Y$ at $\theta$ and will denote it as 
$(X,Y)_\theta$. 
\end{definition}

\begin{nota}\label{notainfty}
Observe that this kind of interpolation can not be applied directly to 
infinite-dimensional spaces unless we consider only equivalent quasinorms on a 
given space. 
\end{nota}

\begin{nota}\label{nota0}
When dealing with vector spaces, it is customary to have clear which vector is 
the origin of the space, in function spaces it is the 0 function, in spaces of 
sequences it is the sequence $(0,0,\ldots)$. But we are giving a vector space 
structure to a set without a clear 0, so we need to choose it. The idea behind 
this work is that we have been given a kind of {\em mean} of two norms in a 
quite intuitive way --for our purposes, the most suitable candidate to be the 
mean of $\nX$ and $\nY$ is $$\norma_{(X,Y)_{1/2}}=\nX^{1/2}\nY^{1/2}.$$
Of course, this means that when we choose the origin of our space, we are given 
the opposite $\norma_{\opX}$ for each $\nX$. The central rôle that the Euclidean 
norm plays in the classical analysis could be enough for it to be our origin, 
but there is another reason for choosing it. When we think of a non strictly 
convex norm, it seems that it is, in some faint sense, an extreme point of a 
segment. A visual way to explain this is the {\em curve} 
$\{[\norma_p]:p\in[1,\infty]\}$. If you reach a non strictly convex norm like 
$[\norma_1]$ or $[\norma_\infty]$ and you keep going in the same direction 
you will find that what you are 
dealing with is not convex any more. In this sense, the Euclidean norm is the 
most convex norm and it deserves to be the centre of our vector space. 
The space $(\R^n,\norma_2)$ is, up to isometric isomorphism, the only homogeneous 
$n$-dimensional space and so, the one with the greatest group of isometries. 
So, we have defined our vector space as follows: 
\end{nota}

\begin{definition}\label{operations}
Let $n\in\N$ and consider $\Q$ as the quotient of the set of quasinorms
on $\R^n$ by the equivalence relation of dilating quasinorms. We consider $\left[\norma_2\right]$
as the origin of our space and the mean of two classes of quasinorms as 
$$([\nX],[\nY])_{1/2}=\left[\nX^{1/2}\nY^{1/2}\right],$$ 
so the opposite of some $[\nX]$ is $[\norma_{\opX}]$, where 
$$\norma_{\opX}=\frac{\norma^2_2}{\nX}$$ 
on $\R^n\setminus\{0\}$ and $\|0\|_{\tilde X}=0$; 
the {\em scalar multiplication} is given by 
$$\theta\star\left[\nX\right]=\left[\norma_X^\theta\norma_2^{1-\theta}\right],\quad 
-\theta\star\left[\nX\right]=\left[\norma_{\wopX}^\theta\norma_{2}^{1-\theta}\right]$$
for $\theta\in[0,\infty)$; and the {\em addition} of two classes of quasinorms by 
$$\left[\nX\right]\oplus\left[\nY\right]=2\star\left[\norma_{(X,Y)_{1/2}}\right].$$ 
\end{definition}

\begin{theorem}\label{thmq}
With the above operations, $\Q$ is a linear space. If we, moreover, define 
$$ \big|\!\big|\!\big| \|\cdot\|_X \big|\!\big|\!\big|=\log_2(d(\|\cdot\|_X,\norma_2)),$$ 
then $\left(\Q,\nnn\cdot\nnn\right)$ is a Banach space where the set of 
equivalence classes of norms in $\R^n$ is closed. 
\end{theorem}

\begin{proof}
It is easy to see that, whenever $\nX$ and $\nY$ are quasinorms over a 
finite-dimensional space $\R^n$ and $\theta>0$, the subset 
$$B_\theta=\{x\in\R^n:\|x\|_X^\theta\|x\|_Y^{1-\theta}\leq 1\}$$
is bounded, absorbing, and balanced, and its boundary is bounded away from 0, so 
Theorem~\ref{rho} implies that $\nX^\theta\nY^{1-\theta}$ is a quasinorm, and 
it is clear that it is continuous. 
So, this kind of {\em extrapolation} of quasinorms is well defined. 
In order to show that $\Q$ is a linear space we need to show that the scalar 
multiplication and the addition are well-defined. On the one hand, it is clear 
that the operations do not depend on the representative of any class of 
quasinorms. On the other hand, all the expressions in Definition~\ref{operations} 
give rise to a continuous quasinorm. 

In~\cite[Theorem~1.18]{Apoorva} it is seen that the distance we are dealing with 
is complete on $\NN$, and this implies that $\NN$ is closed in any metric space 
where it is isometrically embedded, in particular in $\Q$. Anyway, it is not 
hard to see that its complement $\Q\setminus\NN$ is open. 

Now, we need to show that $d$ is absolutely homogeneous and 
additively invariant. 

For the homogeneity, let $\theta\in(0,\infty)$ and take any $\nX$ such that 
$\nX\geq\norma_2$ and $S_X\cap S_2\neq\emptyset$. Then, $(\nX,\norma_2)_\theta$ 
fulfils the same, i.e., $(\nX,\norma_2)_\theta\geq \norma_2$ and 
$S_{(\nX,\norma_2)_\theta}\cap S_2\neq\emptyset$. 
Moreover, if we take $y\in S_2$ such that 
$$d([\nX],[\norma_2])=\|y\|_X,$$
then it is quite clear that 
$$d([(\nX,\norma_2)_\theta],[\norma_2])=\|y\|_X^\theta.$$

For negative values of $\theta$ we only need to see what happens when 
$\theta=-1$, but it is easily seen that 
$d([\nX],[\norma_2])=d([\norma_{\opX}],[\norma_2])$. 

To see that $d$ is additively invariant, take 
$\nX,\nY,\nZ$. We just need to show that 
$$d([(\nX,\nZ)_{1/2}],[(\nY,\nZ)_{1/2}])=d([\nX],[\nY])^{1/2}.$$
For any $z\in\R^n, z\neq 0$ one has 
\begin{equation}\label{XY12}
\frac{\|z\|_{(X,Z)_{1/2}}}{\|z\|_{(Y,Z)_{1/2}}}=
\frac{\|z\|_X^{1/2}\|z\|_Z^{1/2}}{\|z\|_Y^{1/2}\|z\|_Z^{1/2}}=
\left(\frac{\|z\|_X}{\|z\|_Y}\right)^{1/2}.
\end{equation}
Let $\mu=d([\nX],[\nY]).$ 
Applying Lemma~\ref{soloy} we may suppose that $\nX\leq\nY\leq\mu\nX$ and 
choose $x,y$ such that $\|x\|_X=\|x\|_Y=1,$ $\|y\|_X=1/\mu$ and $\|y\|_Y=1.$ 
Now (\ref{XY12}) implies that 
\begin{equation*}
\begin{split}
d([\nX],[\nY])^{1/2}= & \mu^{1/2}= 
\left(\frac{\|x\|_X}{\|x\|_Y}\frac{\|y\|_X}{\|y\|_Y}\right)^{1/2}= 
\frac{\|x\|_{(X,Z)_{1/2}}}{\|x\|_{(Y,Z)_{1/2}}}
\frac{\|y\|_{(X,Z)_{1/2}}}{\|y\|_{(Y,Z)_{1/2}}} \\ 
\leq & d([(\nX,\nZ)_{1/2}],[(\nY,\nZ)_{1/2}])
\end{split}
\end{equation*}
Applying Lemma~\ref{soloy} to $(\nX,\nZ)_{1/2}$ and $(\nY,\nZ)_{1/2}$ 
we see that the symmetric inequality also holds. 

It remains to show the completeness of our norm. 
Take a Cauchy sequence 
$$\big(\big[\norma^1\big],\big[\norma^2\big],\ldots,\big[\norma^k\big],\ldots\big)\subset\Q.$$ 

We may choose a representative of each class, so we may suppose that 
$\norma^k(e_1)=1$ for every $k\in\N.$ 
Every Cauchy sequence is bounded, so we may take $\e, M>0$ such that 
\begin{equation}\label{boundedneigh}
\e\norma_2\leq\norma^k\leq M\norma_2 \text{ for every }k.
\end{equation}
With this in mind, the very definition of $\nnnn$ implies that for every 
$x\in\R^n$ the sequence $\|x\|^k$ is also Cauchy, so we may define 
$\|x\|_X$ as the limit of $\|x\|^k$ as $k\to\infty$. By (\ref{boundedneigh}) 
we have that $B_X$ is a bounded, balanced, neighbourhood of 0, so $\norma_X$ 
is a quasinorm and it is continuous because, locally, it is the uniform limit of 
continuous quasinorms. It is easy see that it is the limit of the sequence, 
and this implies that $\nnnn$ is complete on $\Q$. 
\end{proof}

\begin{remark}\label{Nocvx}
The set of norms is not convex in $\Q$. In fact, if we define 
$\|(a,b)\|_X=2|a|+|b|/2,$ $\|(a,b)\|_Y=2|b|+|a|/2$ then we have 
$$\|(1,0)\|_X=2=\|(0,1)\|_Y,\quad \|(0,1)\|_X=1/2=\|(1,0)\|_Y,$$ 
but $\|(1,1)\|_X=5/2=\|(1,1)\|_Y$, which implies that 
$$\|(1,1)\|_{(X,Y)_{1/2}}>\|(1,0)\|_{(X,Y)_{1/2}}+\|(0,1)\|_{(X,Y)_{1/2}}.$$
\end{remark}

\begin{figure}[ht]
\begin{center}
\begin{tikzpicture}
  \draw[dashed, thick] (0,-2) node[below] {$S_X$} -- (0.5, 0) -- (0,2) -- (-0.5, 0) -- (0,-2);
  \draw[dotted, thick] (-2,0) -- (0, 0.5) -- (2,0) node[below right] {$S_Y$} -- (0, -0.5) -- (-2,0);
  \draw[-, domain=0:1.0, smooth, thick, variable=\y] plot ({\y}, {(-17*\y+sqrt(225*\y*\y+64))/8});
  \draw[-, domain=0:1.0, smooth, thick, variable=\y] plot ({-\y}, {(-17*\y+sqrt(225*\y*\y+64))/8});
  \draw[-, domain=0:1.0, smooth, thick, variable=\y] plot ({-\y}, {-(-17*\y+sqrt(225*\y*\y+64))/8});
  \draw[-, domain=0:1.0, smooth, thick, variable=\y] plot ({\y}, {-(-17*\y+sqrt(225*\y*\y+64))/8});
\end{tikzpicture}
\caption{The spheres of the three quasinorms of Remark~\ref{Nocvx}.} 
\end{center}
\end{figure}
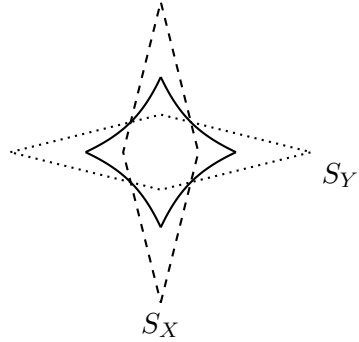

\begin{nota}\label{Pn}
We can describe the space $\Q$ as some $C(K)$. 
Namely, let $\P_{n-1}$ be the projective space of dimension $n-1$, i.e., 
$$\P_{n-1}=\big(\R^n\setminus\{0\}\big)/\coc, \text{ with } x\sim y \text{ if and only if }
x=\lambda y \text{ for some }\lambda\in\R\setminus\{0\},$$ 
endowed with the quotient topology relative to the projection 
$\R^n\setminus\{0\}\to \big(\R^n\setminus\{0\}\big)/\coc,\ x\mapsto[x]$.
In the sequel, we will think the projective space of dimension $n-1$ as the 
quotient $S^{n-1}/\coc$, where $x\sim y$ if and only if $x=\pm y$ and $S^{n-1}$ 
denotes the sphere of $(\R^n,\norma_2)$. 

This is essentially the same idea as that reflected in Subsection~3.2 
({Proof of the main result}) in~\cite{Apoorva}. 

As every quasinorm is absolutely homogeneous, i.e., $\|\lambda x\|=|\lambda|\|x\|$ 
for any $\lambda\in\R, x\in X$, $\norma$ is always determined by its value at every 
point of, say, the Euclidean sphere $S^{n-1}$. 
Take into account now the universal property of the quotient, that assures that any 
continuous function $f:S_X\to\R$ 
such that $f(x)=f(-x)$ gives rise to a well-defined and continuous 
$\tilde f:\P_{n-1}\to\R,\ \tilde f([x])=f(x)$. With this in mind, it is clear that each 
continuous quasinorm $\nX$ defines a continuous function $f_X:\P_{n-1}\to(0,\infty)$. 

Recall that $\P_{n-1}$ is compact --it is the continuous image of a compact 
space--, so every continuous 
$f:\P_{n-1}\to(0,\infty)$ is bounded from above and bounded away from 0, and 
we can define a quasinorm on $\R^n$ as $\|\lambda x\|_f=|\lambda|f([x])$ for every 
$\lambda\in\R,\ x\in S_2$. 

It is clear that this is a one-to-one correspondence between the space of continuous 
quasinorms $\norma:\R^n\to\R$ and the space of positive continuous functions 
$\P_{n-1}\to(0,\infty)$. If we consider again the equivalence relation 
$\nX\sim\nY\iff\nX=\lambda\nY$, $\lambda\in\R\setminus\{0\}$, then the correspondence 
still holds if we consider $C(\P_{n-1})$ endowed with the equivalence relation 
$f\equiv g\iff f=\lambda g$, $\lambda\in\R\setminus\{0\}$. So, we have a bijection 
$\Q\bij C(\P_{n-1},(0,\infty))/\!\equiv$. To end the description of $\Q$ we just need 
to consider $\log:C(\P_{n-1},(0,\infty))\to C(\P_{n-1})$, endow this space with 
the equivalence relation $f\sim g\iff f=\lambda+g$ for some $\lambda\in\R$ 
--to preserve the bijection with the former space-- and, for any 
$[f],[g]\in C(\P_{n-1})/\coc$ define the metric 
$$d([f],[g])=\max_{x\in\P_{n-1}}\{f(x)-g(x)\}-\min_{x\in\P_{n-1}}\{f(x)-g(x)\}, \text{ where } f,g 
\text{ are any representatives of }[f], [g]. $$
Observe that this value is the range of $f-g$ and that we can rewrite this as 
$$d([f],[g])=\max_{x\in\P_{n-1}}\{f(x)-g(x)\}+\max_{x\in\P_{n-1}}\{g(x)-f(x)\}, \text{ where } f,g 
\text{ are any representatives of }[f], [g]. $$ 
With this, we have $[\norma_2]\in\Q\mapsto [0]\in C(\P_{n-1})/\coc$
(see, again,~\cite{Apoorva}, Subsection~3.2). 
We also have that the map between $(\Q,\nnnn)$ and $C(\P_{n-1})/\coc$ is an onto 
isometry. Moreover, if $e_1$ is the first vector 
of the usual basis of $\R^n$, then we can see the latter space as 
$$C_0(\P_{n-1})=\{f\in C(\P_{n-1}):f([e_1])=0\}, $$
whose bijection with some space of quasinorms arises from considering only the 
quasinorms in $\Q_0$ that take value 1 at $e_1$. 

The reader interested in Projective Geometry can check out ~\cite{coxeter, richter}. 
\end{nota}

\section{The Banach-Mazur compactum}

\begin{nota}\label{matrices}
The space of $n\times n$ real matrices 
will be denoted as $\mathcal{M}_n$. 

Every time we write {\em isometry} we will mean {\em linear isometry}. 
This, by the Mazur-Ulam Theorem, means just that we will consider only isometries 
sending 0 to 0. 

As we will deal just with finite-dimensional spaces, we can fix the standard 
basis of $\R^n$, so that each operator $T:\RnX\to\RnY$ can be seen as a matrix 
$A\in\mathcal{M}_n$. We will use $AB_X=\{Ax:x\in B_X\}$ and $TB_X$ indistinctly. 

We will intertwine operators and norms and will need some notation for 
{\em the norm whose value at each $x$ is $\|Ax\|_X$ (resp. $\|Tx\|_X$)}, 
where $A\in\GL(n)$ (resp. $T$ is a linear isomorphism). This will be written as 
$A^*\nX$ (resp. $T^*\nX$).
\end{nota}

Now that we have determined the structure of $\Q$, we may relate it to the 
well known Banach-Mazur compactum. This compactum is obtained 
by endowing the set $\mathcal{N}_0$ of norms defined on $\R^n$ with the pseudometric 

$$d_{BM}(\nX,\nY)=\min\left\{\big\|T\big\|\big\|T^{-1}\big\|\right\}, $$
where the minimum is taken in 
$$\{T:(\R^n,\nX)\to(\R^n,\nY)\text{ is a linear isomorphism}\}.$$ 

This pseudometric does not distinguish between isometric norms, so the quotient 
needed to turn it into a metric is by the equivalence relation 
$$\nX\equiv\nY \text{ when there is a linear isometry }T:(\R^n,\nX)\to(\R^n,\nY).$$ 
As we are dealing with finite-dimensional spaces, 
the isomorphism $T$ can be seen as an invertible matrix of order $n$, i.e., 
$T$ is associated to some $A\in\GL(n)$. Conversely, every invertible matrix 
gives an isomorphism, so the Banach-Mazur distance can be seen as 
$$d_{BM}(\nX,\nY)=\min\{\mu:B_X\subset AB_Y\subset \mu B_X\text{ for some }A\in \GL(n)\}$$ 
and the quotient as  
$$\nX\equiv\nY \text{ if and only if there is } A\in\GL(n) 
\text{ such that } AB_X=B_Y.$$

As the equivalence relation $\sim$ that defines $\Q$ can obviously be seen as 
$$\nX\sim\nY \text{ if and only if there is } 
\lambda\neq 0 \text{ such that } \lambda B_X=B_Y,$$  
if we denote $\R^*=\R\setminus\{0\}$, then the relation between both spaces 
seems to be given by $\PGL(n,\R)=\GL(n)/\R^*$. It is, however, a little more complex. 

Let us study the fibers in $\NN=\{\text{Norms\ defined\ on\ }\R^n\}/\coc$ 
of each element of $BM=(\NN_0/\equiv)=(\NN/\equiv)$. 
Suppose we are given a norm $\nX$ whose group of autoisometries is trivial, 
i.e., the only (linear) isometries $\RnX\to\RnX$ are the identity and its 
opposite. Then, $AB_X= CB_X$ implies $A=\pm C$ and this means that the fiber 
of $[\nX]\in BM$ in $\NN$ is indeed $\{A^*\nX:A\in\GL(n)\}/\R^*$. 
However, if $\nX$ has nontrivial group of autoisometries then $AB_X=AGB_X$ 
whenever $G:\RnX\to\RnX$ is an isometry. Denoting as $\Iso_X$ this group of 
autoisometries for each $\nX$ we obtain a one-to-one relation
$$\NN \longleftrightarrow\{(\{[\nX]\}\times\PGL(n,\R))/\Iso_X:[\nX]\in BM\}.$$

Before we proceed with the main result in this section we need a couple of 
results about the group $\Iso_X$. As is customary, the distance between two 
linear operators $F,G:(\R^n,\norma_X)\to(\R^n,\norma_Y)$ is defined as the 
operator norm of its difference: 
$$d(F,G)=\|F-G\|_Y=\max\{\|F(x)-G(x)\|_Y:x\in B_X\}.$$

The first result we need is as follows:
\begin{lemma}\label{lemaaxler}
Let $F:(\R^n,\norma_X)\to(\R^n,\norma_X)$ be an isometry. Then there are linearly 
independent $u,v\in \R^n$ such that the plane $\langle u,v\rangle$ is invariant 
for $F$ and such that the matrix of the restriction of $F$ to $\langle u,v\rangle$ 
with respect to the basis $\{u,v\}$ is one of the following: 
\begin{equation}
\left(
\begin{array}{c c}
\cos(\alpha)&-\sin(\alpha)\\
\sin(\alpha)&\cos(\alpha)
\end{array}
\right), 
\qquad
\left(
\begin{array}{c c}
1 & 0\\
0 & -1
\end{array}
\right),
\end{equation}
where $\alpha\in(-\pi,\pi]$. 
\end{lemma}

\begin{proof}
If $F=\pm \Id$ then the results holds with $\alpha=0,\pi$, so we assume 
henceforth that this is not the case. It is well known that every linear 
endomorphism $F:(\R^n,\norma_X)\to(\R^n,\norma_X)$ has at least one complex 
eigenvalue $\lambda$ (see, e.g., \cite[9.8]{axler}) and non-real eigenvalues 
occur in conjugate pairs. If $\lambda=a+bi\not\in\R$, then there are 
$u,v\in X\setminus\{0\}$ such that $F(u)=au-bv$ and $F(v)=av+bu$. Let 
$|\lambda|=\sqrt{a^2+b^2}$ and $H$ be the plane generated by $u$ and $v$ 
endowed by the basis $\{u,v\}$ --observe that 
$F(H)=H$. Then, there is 
$\alpha\in(-\pi,\pi)$ such that the matrix of the restriction of $F$ to $H$ is 
\begin{equation*}
|\lambda|
\left(
\begin{array}{c c}
\cos(\alpha)&-\sin(\alpha)\\
\sin(\alpha)&\cos(\alpha)
\end{array}
\right).
\end{equation*}
As $F$ is an isometry, one has $\|F^k(u)\|_X=1$ for every $k\in\N$, so the 
sequence $(F^k(u))_k$ is bounded (with respect to every norm) and it is clear 
that this implies that $|\lambda|=1$.

Suppose, now, that every eigenvalue is real, and let $\lambda\in\R$ and $u\in S_X$ 
be such that $F(u)=\lambda u$. It is obvious that, again, $|\lambda|=1$. If $F$ has at least two different eigenvalues then we may suppose $\lambda=1$ and the other 
eigenvalue must be $-1$, so let $v$ be such that $F(v)=-v$. With respect to the basis 
$\{u,v\}$ the matrix of the restriction of $F$ to $\langle u,v\rangle$ is 
\begin{equation*}
\left(
\begin{array}{c c}
1 & 0\\
0 & -1
\end{array}
\right).
\end{equation*}

This leaves the case where all the eigenvalues of $F$ are real and are the same. 
We may suppose that all 
of them equal 1. As we are assuming that $F\neq\Id$, it is clear that the dimension 
of $\ker(F-\Id)$ is at most $n-1$ and the Cayley-Hamilton Theorem implies that 
$\dim(\ker(F-\Id)^2)>\dim(\ker(F-\Id))$. This means that we may find $v\in S_X,$ 
$u\in X\setminus\{0\},$ such that $u=F(v)-v\neq 0,$ and $ F(u)-u=(F-\Id)(u)=(F-\Id)^2(v)=0$. 
Thus, we have $F(v)=u+v$ and $F(u)=u$ and this implies that in the plane 
$\langle u,v\rangle$ endowed with the basis $\{u,v\}$, the matrix of the 
restriction of $F$ is 
\begin{equation}
\left(
\begin{array}{c c}
1 & 1\\
0 & 1
\end{array}
\right),
\end{equation}
which leads to the matrix of $F^k$:
\begin{equation}
\left(
\begin{array}{c c}
1 & k\\
0 & 1
\end{array}
\right). 
\end{equation}
So, the sequence $(F^k(v))_k=(ku+v)_k$ is unbounded and we are done. 
\end{proof}

\begin{remark}
It is easy to see that the same computation as the one at the end of the proof of 
Lemma~\ref{lemaaxler} rules out the option that $F$ has some Jordan block of the form: 
\begin{equation}
\left(
\begin{array}{c c c c}
\cos(\alpha) & -\sin(\alpha) & 1 & 0 \\
\sin(\alpha) &  \cos(\alpha) & 0 & 1 \\
0 & 0 & \cos(\alpha) & -\sin(\alpha) \\
0 & 0 & \sin(\alpha) &  \cos(\alpha) 
\end{array}
\right),
\end{equation}
so every isometry is diagonalizable over $\C$. 
\end{remark}

\begin{remark}
We have not used the fact that $F$ is an isometry, we merely needed that the 
sequence $(\|F^k\|_X)_k$ is bounded and bounded away from 0. 
\end{remark}

\begin{proposition}\label{notrivial}
Whenever $\norma_X$ has nontrivial group of isometries, there is some 
autoisometry $F:(\R^n,\norma_X)\to(\R^n,\norma_X)$ 
such that $\min\{\|F+\Id\|_X,\|F-\Id\|_X\}\geq 1$. 
\end{proposition}

\begin{proof}
Let $F\in\Iso_X, F\neq\pm\Id$. Then 
$\max\{\|F(x)+x\|_X,\|F(x)-x\|_X\}\leq 2$ for every $x\in S_X$, so 
$\max\{\|F+\Id\|_X,\|F-\Id\|_X\}\leq 2$. 
If all the eigenvalues of $F$ are real, then the proof of Lemma~\ref{lemaaxler}
implies that there are $u,v\in S_X$ such that $\|F(u)+u\|_X=2, \|F(v)-v\|_X=2$, 
so we actually have $\|F+\Id\|_X=\|F-\Id\|_X=2$. 

If some eigenvalue is not real, say $\lambda=a+bi, b\neq 0$, then we know 
by Lemma~\ref{lemaaxler} that $|\lambda|=1$. Let $u\in S_X,v\in X\setminus\{0\}$ 
be such that the matrix of the restriction of $F$ to $H=\langle u,v\rangle$ is, 
with respect to the basis $\{u,v\}$, the rotation of angle $\alpha$
$$\left(
\begin{array}{c c}
\cos(\alpha)&-\sin(\alpha)\\
\sin(\alpha)&\cos(\alpha)
\end{array}
\right)$$ 
for some $\alpha\in(-\pi,\pi]$ --the existence of such a basis is outlined in the proof of Lemma~\ref{lemaaxler}. If $\alpha>\pi/2$ (respectively 
$\alpha\leq -\pi/2$) then we may compose with $-\Id$ and get the rotation 
of angle $-\pi+\alpha$ (respectively $\pi+\alpha$), so 
we may suppose $\alpha\in(-\pi/2,\pi/2]$. If $\alpha<0$ then the inverse of 
$F_{|H}$ is the rotation of angle $-\alpha$, so we only need to deal with 
$\alpha\in [0,\pi/2]$. As $\alpha=0$ gives the identity, 
what we have is $\alpha\in (0,\pi/2]$. Now we have to break down 
the different options. 

If $\alpha=\pi/m$ for some $m\in\N$, then $F_{|H}^m=-\Id_{|H}$. Consider 
the {\em half-orbit} of $u$, $\{x_0=u,x_1=F(u),\ldots,x_m=F^m(u)=-u\}$. 

If $m\in 2\mathbb{Z}$, then $x_{m/2}=v$ is at the same distance from $u$ and 
$-u$ because $F^{m/2}(u)=v$ and $F^{m/2}(v)=-u$. Indeed, 
$$\|v-u\|_X=\|F^{m/2}(u)-u\|_X=\|F^{m/2}(F^{m/2}(u)-u)\|_X=\|-u-v\|_X.$$
This readily implies that 
\begin{equation}\label{NoEstr}
\min\{\|F^{m/2}-\Id\|_X, \|F^{m/2}+\Id\|_X\}\geq \|v-u\|_X=\|v+u\|_X\geq 1,  
\end{equation}
where the last inequality holds because of the triangular inequality: 
$$2=\|v+v\|_X\leq \|v-u\|_X+\|v+u\|_X=2\|v-u\|_X.$$ 

For $m\not\in 2\N$ we are going to restrict every coordinate-wise computation 
to the plane $H$ --the difference would be a certain amount of zeroes after 
the two first coordinates. 
Taking coordinates with respect to $\{u,v\}$, we have 
$$x_{(m-1)/2}=\left(\cos\left(\frac{(m-1)\pi}{2m}\right)\!\!, 
\sin\left(\frac{(m-1)\pi}{2m}\right)\right)\!,\
  x_{(m+1)/2}=\left(\cos\left(\frac{(m+1)\pi}{2m}\right)\!\!, 
\sin\left(\frac{(m+1)\pi}{2m}\right)\right).$$ 

We are going to show that $\|x_{(m-1)/2}-u\|_X\geq 1$. For this, we first need 
to check that the segment whose endpoints are $x_{(m\pm 1)/2}$ equals the 
intersection of the line that contains both of them with the unit ball $B_X$.
Observe that the first coordinate of $x_{(m+1)/2}$ is the opposite of the first 
coordinate of $x_{(m-1)/2}$ and that the second coordinates of both points agree. 
So, if we denote by $r$ the horizontal line whose height is $\sin((m+1)\pi/2m)$, 
we have $x_{(m\pm 1)/2}\in r\cap S_X$. The convexity of $B_X$ implies that if 
there are three collinear points in $S_X$, then the segment determined by them 
is included in $S_X$, too. In particular, if there is some 
$y\in (S_X\cap r)\setminus\{x_{(m\pm 1)/2}\}$, 
then the segment whose endpoints are $x_{(m\pm 1)/2}$ is included in $S_X$. 
On the one hand, this means that the {\em Euclidean} regular $2m$-gon with 
vertices in every $x_k$ is included in $S_X$ because each segment of the 
$2m$-gon is the image of this segment by some $F^k$. 
On the other hand, under these circumstances it is clear that this $2m$-gon 
is $S_X\cap H$, so in any case, $(t,\sin((m-1)\pi/2m))\in B_X$ if and only if 
\begin{equation}\label{segmento}
t\in[-\cos((m-1)\pi/2m)),\cos((m-1)\pi/2m))].
\end{equation}

If $m=3$, then $x_1=(1/2,\sqrt{3}/2)$ and $x_2=(-1/2,\sqrt{3}/2)$, so 
$x_1-u=x_2$. This implies that $\|x_1-u\|_X=\|x_2+u\|_X=1$, and also 
$\|x_1+u\|_X=\|x_2-u\|_X>1$. So, $\min\{\|F+\Id\|_X,\|F-\Id\|_X\}\geq 1$. 
 
If $m\geq 5$, then $0<\cos((m-1)\pi/2m))<\cos(\pi/3)=1/2$ and this, along with 
(\ref{segmento}), implies that 
$$x_{(m-1)/2}-u=(\cos((m-1)\pi/2m))-1,\sin((m-1)\pi/2m)))$$ 
lies outside the unit ball, so $\|F^{(m-1)/2}\pm\Id\|_X>1$. 

If $\alpha=\frac pq\pi$ for some coprime $p,q\in\N$, then the Chinese Remainder 
Theorem implies that the rotation of angle $\pi/q$ is also an isometry and we 
are in the previous case. 

If $\alpha\not=\frac pq\pi$ for any $p,q\in\N$, then the orbit of $u$ is dense 
in $S_H$ and, actually, in the sphere of $\norma_2$, i.e., in 
$\{\lambda u+\mu v\in H:\lambda^2+\mu^2=1\}$. The continuity of $\nX$ with respect 
to any norm defined over $H$ implies that $\nX$ restricted 
to $H$ is $\norma_2$ and so, any map $F^k$ that sends $u$ close enough to $v$ 
has distance to $\pm\Id$ close to $\sqrt{2}>1$. 
\end{proof}

\begin{definition}
Let $\nX$ be a norm defined over $\R^n$. We say that $\nX$ is a polyhedral norm 
or, equivalently, that $(\R^n,\nX)$ is a polyhedral space, if its closed unit ball 
is a polytope. 
\end{definition}

\begin{definition}
Given a normed space $\XnX$, we say that $x\in S_X$ is an exposed point if there 
is $f\in X^*$ such that $f(x)=1$ and $f(y)<1$ for every $y\in S_X, y\neq x$. 
We say that $x\in S_X$ is an extreme point if it does not lie in the interior 
of a segment included in $S_X$. 
\end{definition}

\begin{nota}
It is clear that if $B_X$ is a polytope, then $x\in B_X$ is extreme if and only if 
it is exposed. 
\end{nota}

We will need the following weak version of the Krein-Milman Theorem, see~\cite{KM}: 
\begin{theorem}[Krein-Milman]
The unit ball of every finite dimensional normed space is the convex hull of 
its subset of extreme points. 
\end{theorem}

In the proof of Theorem~\ref{denso} we will also use the Brouwer fixed-point 
Theorem, see~\cite[Theorem~6]{90B} or directly,~\cite{schauder}: 

\begin{theorem}
If $C$ is a closed convex subset of a Banach space, then every compact 
continuous map $f : C\to C$ has a fixed point.
In particular, if $C$ is convex and compact, then every continuous 
map $f : C\to C$ has a fixed point. 
\end{theorem}

Now we can proceed with the main result in this section. 

\begin{theorem}\label{denso}
Let $U=\{[\nX]\in\NN:\Iso_X=\{\Id,-\Id\}\}.$
Then, $U$ is a dense open subset of $\NN$. 
\end{theorem}

\begin{proof}
To see that $U$ is dense we need the following fact: \\
The subset of equivalence classes of polyhedral norms is dense in $\NN$. 
This is clear from \cite[Theorem~1.1]{DFH}. 

With this fact in mind, and given some polyhedral norm $\nX$, we are going to 
sketch how to construct a norm with trivial group of isometries and whose distance 
to $\nX$ is as small as we want. The Krein-Milman Theorem implies that there is a 
basis $\B=\{x_1,\ldots,x_n\}$ such that every $x_i$ is an exposed point of $B_X$. 
Given $\delta>0$ we may consider 
$$x_{n+1}=\frac{1+\delta}{\|(1,\ldots,1)\|_X}(1,\ldots,1)$$
and the norm $\norma_{X'}$ whose unit ball is the convex hull of 
$B_X\cup\{\pm x_{n+1}\}$. On the one hand, this norm is as close as we want to 
$\nX$, so we just need to approximate $\norma_{X'}$. On the other hand, 
every $x_i$ with $i=1,\ldots,n+1$ is exposed in $B_{X'}$. Indeed, 
for each $i\in\{1,\ldots,n\}$, consider some linear $f_i:\R^n\to\R$ such 
that $f_i(x_i)=1$ and $f_i(y)<1$ for every $y\in S_X, y\neq x_i$.
It is clear that there exist $\alpha_1,\ldots,\alpha_n\in(0,\infty)$ such that 
$f_i(x_{n+1})<(1-\alpha_i)$. As there are finitely many $\alpha_i$ we can choose
$\delta>0$ so that $f_i(x_{n+1})<(1-\alpha_i)(1+\delta)<1$ for every $i=1,\ldots,n$. 
This means that, when $\delta>0$ is small enough, $f_i(x_{n+1})<1$. The only 
points in $B_{X'}$ that do not belong to $B_X$ are $x_{n+1}$ and convex combinations 
$\lambda x_{n+1}+(1-\lambda)z$ with $z\in B_X$ and $\lambda\in(0,1)$. This 
clearly implies that $f_i(y)<1$ for every $y\in B_{X'}$, $y\neq x_i$, with 
$i=1,\ldots,n.$

Now, consider some linear $f_{n+1}:\R^n\to\R$ such 
that $f_{n+1}(x_{n+1})=1$ and $f_{n+1}(y)<1$ for every $y\in S_X, y\neq x_{n+1}$. Choose a 
basis $\B^i=\{u^i_1,\ldots,u^i_n\}\subset S_X$ with $u^i_1=x_i$ and 
$u^i_j\in\ker(f_i)$ when $j\neq i$. Given some $M_i>0$ and $1>\e_i>0$ we may define 
$\norma_i$ as 
$$\|\lambda_1u^i_1+\ldots+\lambda_{n}u^i_{n}\|_i=\big(|(1+\e_i)\lambda_1|^{2i+2}+
(|\lambda_2|/M_i)^{2i+2}+\ldots+(|\lambda_{n}|/M_i)^{2i+2}\big)^{1/{(2i+2)}}.$$
If we take $\e_i$ small enough and $M_i$ big enough, then the norm 
$\nY=\max\{\norma_{X'},\norma_1,\ldots,\norma_n\}$ equals 
$\norma_{X'}$ in every point of $S_{X'}$ except for small neighbourhoods of 
$(1-\e)x_1,\ldots,(1-\e)x_{n+1}$, say $V_1,\ldots,V_{n+1}$, where the sphere 
takes the form of a variant of the $p$-norm with $p={2i+2}$. Observe that we 
may take each $e_{i+1}$ small enough and $M_{i+1}$ big enough to make the 
diameter of $V_{i+1}$ strictly smaller than that of $V_i$ and, moreover, 
we may suppose that the diameter of every $V_i$ is strictly smaller than 
the distance between any pair of $V_j, V_k$, with $i,j,k\in 1,\ldots,{n+1}$. 
\begin{claim*}
Reducing $\e$ and increasing $M$ if necessary, we may suppose 
that  every collection $y_1\in V_1,\ldots, y_{n+1}\in V_{n+1}$ are in general 
position, i.e., no hyperplane contains $n$ of them. 
\end{claim*}
\begin{proof}
This is clear from the following facts:
\begin{enumerate}
\item
$x_1,\ldots,x_{n+1}$ are in general position. 
\item 
An $n$-tuple $\{u_1,\ldots,u_n\}$ lie in the same hyperplane if and only if 
every skew-symmetric linear $n$-form vanishes when applied to it, i.e., 
$\omega(u_1,\ldots,u_n)=0$ for every (some) $\omega:(\R^n)^n:\to\R, \omega\neq 0.$ 
\item Any skew-symmetric linear $n$-form is continuous. 
\end{enumerate}
\end{proof}
This new norm $\nY$ has trivial autoisometry group. Indeed, the points in 
$V_1,\ldots,V_{n+1}$ are the only exposed points where $S_{Y}$ is smooth 
--besides $-V_1,\ldots,-V_{n+1}$. 
So, as being exposed and being smooth are properties preserved by linear isometries, 
$(\bigcup V_i)\bigcup(\bigcup -V_i)$ is invariant for any autoisometry 
$F:\RnY\to\RnY$. There is no way that $(V_i,\norma_{Y})$ is isometric to 
$(V_j,\norma_{Y}), j\neq i,$ because their diameters are different, so every 
$V_i\bigcup(-V_i)$ is invariant for $F$. 

Let us denote by $\ch(V_i)$ the convex hull of $V_i$, analogously 
$\ch(-V_i)$. As $F$ is linear, $\ch(V_i)\bigcup\ch(-V_i)$ is invariant 
for $F$, too. Now, either $F$ or $-F$ sends $V_i$ onto itself. Thus, the 
Brouwer fixed-point Theorem implies that either $F$ or $-F$ has some 
fixed point $y_i\in\ch(V_i)$. In any case, $\{F(y_i),F(-y_i)\}=\{y_i,-y_i\}$ 
for every $i=1,\ldots,n+1$. 

So, in the basis $\{y_1,\ldots,y_n\}$, the matrix of $F$ is diagonal, 
and all the diagonal entries are $\{\pm 1\}$, say the $k$-th is $\delta_k$. 
In this basis, we have $y_{n+1}=(\lambda_1,\ldots,\lambda_n)$, with 
$\lambda_1\cdots\lambda_n\neq 0$ --recall that $\{y_1,\ldots,y_{n+1}\}$ 
are in general position-- and $\{\pm y_{n+1}\}$ is also invariant, say 
$F(y_{n+1})=\delta_{n+1}y_{n+1}$. As $F$ is linear we have 
$$\delta_{n+1}(\lambda_1,\ldots,\lambda_n)=\delta_{n+1}y_{n+1}= 
F(\lambda_1,\ldots,\lambda_n)=(\delta_1\lambda_1,\ldots,\delta_n\lambda_n).$$ 
So, $\delta_1=\cdots=\delta_n=\delta_{n+1}$ and this 
implies that $F$ is either the identity or $-\Id$. 

With $\e$ close enough to 0 and $M$ great enough, $\nX'$ is as close to $\nX$ 
as we want, so $U$ is dense. 

To show that $U$ is open, let $([\norma^k])_k\subset U^c$ be a convergent 
sequence. We need to show that $[\norma]=\lim([\norma^k])$ has nontrivial group of 
isometries, i.e., that $U^c$ is closed. As the sequence of norms converges, in 
particular it is bounded, so there exists $R\in(1,\infty)$ such that 
$d([\|\cdot\|]^k,[\|\cdot\|]_2)\leq R$ for every $k\in\N$. So, for each $k$, we 
may take representatives $\norma^k$ such that 
\begin{equation}\label{2kR2}
\norma_2\leq\norma^k\leq R\norma_2,
\end{equation}
and also $\norma_2\leq\norma\leq R\norma_2$. 
Suppose that for every $\norma^k$ there exists $T_k\in\Iso_{X_k} \setminus\{\Id,-\Id\}$. 
By~\eqref{2kR2}, $1\leq\|T_kx\|_2\leq R$ for every $k\in\N$ and $x\in S_{X_k}$, 
so $(T_k)$ is uniformly bounded in $\mathcal{M}_n$ endowed with 
the Euclidean matrix norm. This implies that  $(T_k)_k$ 
must have some accumulation point $T$; we will suppose that $T$ is the limit of 
the sequence. We need to see that $T$ is an autoisometry for $\norma$ and that 
it can be chosen to be neither $\Id$ nor $-\Id$. 

For the first part, applying the triangle inequality to $\nnnn$ gives 
\begin{equation}\label{eqnnnnnT}
\begin{split}
\nnn T^*\norma-\norma\nnn & \leq 
\nnn T^*\norma-T^*_k\norma\nnn 
+\nnn T^*_k\norma-T^*_k\norma^k\nnn \\
&+\nnn T^*_k\norma^k-\norma^k\nnn 
+\nnn \norma^k-\norma\nnn. 
\end{split}
\end{equation}
The third term in the sum is 0 for every $k$ and the fourth term tends to 0 when 
$k\to\infty$, so we need to show that it is also the case for the first two terms, or, 
equivalently, that the map 
$$(T,\norma) \mapsto T^*\norma$$ 
--that assigns to each linear operator $T:\R^n\to\R^n$ and each norm $\norma:\R^n\to\R$
the norm $T^*\norma$ defined as $T^*\|x\|=\|Tx\|$-- is continuous. 
For the sake of clarity, we will denote $S=S_{X_k}$ for the remainder of the proof. 
We need to show that 
$$\lim_k\max_{y\in S}\left\{\frac{T^*_k\|y\|}{T^*\|y\|}\right\}
\max_{y\in S}\left\{\frac{T^*\|y\|}{T^*_k\|y\|}\right\}=1.$$ 
Given $k\in\N$, Lemma~\ref{soloy}, implies that we may take $y_k, z_k$ such that 
$$\max_{y\in S}\left\{\frac{T^*_k\|y\|}{T^*\|y\|}\right\}
\max_{y\in S}\left\{\frac{T^*\|y\|}{T^*_k\|y\|}\right\}=
\frac{T^*_k\|y_k\|}{T^*\|y_k\|}\frac{T^*\|z_k\|}{T^*_k\|z_k\|}$$
and one has
$$\lim_k\max_{y\in S}\left\{\frac{T^*_k\|y\|}{T^*\|y\|}\right\}
\max_{y\in S}\left\{\frac{T^*\|y\|}{T^*_k\|y\|}\right\}=
\lim_k\frac{T^*_k\|y_k\|}{T^*\|y_k\|}
\frac{T^*\|z_k\|}{T^*_k\|z_k\|}
=\lim_k\frac{\|T_ky_k\|}{\|Ty_k\|}
\frac{\|Tz_k\|}{\|T_kz_k\|}=1$$
since $\norma$ is continuous. Analogously we see that 
$$\lim_k\max_{y\in S}\left\{\frac{T^*_k\|y\|^k}{T_k^*\|y\|}\right\}
\max_{y\in S}\left\{\frac{T^*_k\|y\|}{T_k^*\|y\|^k}\right\}=1.$$ 

So, taking logarithms, the right hand side of the inequality (\ref{eqnnnnnT}) 
converges to 0 and this implies that $\nnn T^*\norma-\norma\nnn=0$, so 
$T^*\norma=\norma$ and $T$ is an isometry. 

Proposition~\ref{notrivial} implies that we can choose every $T_k$ at distance 
at least 1 from $\pm\Id$, so $T\neq\pm\Id$ and we are done.
\end{proof}

\begin{remark}
In the previous proof we have seen that $(T,\norma)\mapsto T^*\norma$ is continuous 
when $\norma$ is a norm. This is not always true when $\norma$ is a quasinorm. 
Indeed, we just need to consider $\R^2$ endowed with the quasinorm 
$$\|(x,y)\|=\left\{
\begin{array}{r l}
\|(x,y)\|_2 & \text{ if } (x,y)\not\in\{(\lambda,0):\lambda\in\R^*\} \\
\frac 12\|(x,y)\|_2 & \text{ if } (x,y)\in\{(\lambda,0):\lambda\in\R\}
\end{array}
\right.,$$
define the operators
$$T_k(x,y)=\left(\!\!
\begin{array}{c r}
\cos(\pi/k) & -\sin(\pi/k) \\
\sin(\pi/k) & \cos(\pi/k)
\end{array}
\!\!\right)
\left(\!\!
\begin{array}{c }
x \\
y
\end{array}
\!\!\right)
$$
and observe that $\nnn T^*_k\norma-T^*_l\norma\nnn$ does not depend on $k, l\in\N$ 
as long as they are different. Indeed, the operator $T_k$ is the rotation of angle 
$\pi/k$ and the only points in the sphere of $\norma$ outside the 
Euclidean sphere are $\pm(2,0)$. So, $T^*_k\|x\|=\|x\|_2$ for every $k\in\N$ 
unless $T_k(x)=(\lambda,0)$, in which case $T^*_k\|x\|=\|x\|_2/2.$ So, 
if $k$ and $l$ are different then one has 
$$\nnn T^*_k\norma-T^*_l\norma\nnn=
\log_2(d(T^*_k\norma,T^*_l\norma))=
\log_2\left(\max_{x\in S}\frac{\|T^k(x)\|}{\|T^l(x)\|}
\max_{x\in S}\frac{\|T^l(x)\|}{\|T^k(x)\|}\right)=
\log_2(4)=2, $$
where $S$ denotes the unit sphere of the norm $\norma.$

In spite of this, it is quite clear that the proof of the continuity 
of $(T,\norma)\mapsto T^*\norma$ still 
works when we deal with continuous quasinorms. 
\end{remark}

\section*{Acknowledgements}
This work has been partially supported by 
Junta de Extremadura programs GR-15152 and  IB-16056 and DGICYT 
projects MTM2016-76958-C2-1-P and PID2019-103961GB-C21 (Spain).
The second author was supported by the grant BES-2017-079901 related to the 
project MTM2016-76958-C2-1-P.

We are greatly indebted to Professors Jesús Castillo and Apoorva Khare for some 
discussions on this topic. 
It is also a pleasure to thank the referees for their careful reading, their 
deep comments about this paper and their patience. 

\bibliographystyle{abbrv}

\bibliography{NormedNorm}

\end{document}